\documentclass[a4paper,12pt]{amsart}

\usepackage{pdfsync}
\def\usealternativegraphics{false}
\usepackage{verbatim,ifthen}
\usepackage{amsmath,amssymb,latexsym,ifpdf, color}

\usepackage{pgf}

\ifthenelse{\equal{\usealternativegraphics}{false}}{\usepackage{tikz} 
\usetikzlibrary{calc}}{}

\ifthenelse{\equal{\usealternativegraphics}{true}}{
\newenvironment{alternativegraphic}[2][1]{
\ifthenelse{\equal{none}{#2}}{
  \textbf{Picture omitted}
}{
  \ifpdf
    \IfFileExists{#2.pdf}{\includegraphics[page=#1]{#2.pdf}}{\textbf{PDF file not found.}}
  \else
    \IfFileExists{#2.eps}{\includegraphics{#2.eps}}{
      \IfFileExists{#2_#1.eps}{\includegraphics{#2_#1.eps}}{\textbf{EPS file not found.}}
    }
  \fi
}
\comment
}{
\endcomment
}
}{
\newenvironment{alternativegraphic}[2][]{}{}
}

\numberwithin{equation}{section}

\newtheorem{thm}{Theorem}[section]
\newtheorem{prop}[thm]{Proposition}
\newtheorem{lemma}[thm]{Lemma}
\newtheorem{cor}[thm]{Corollary}

\theoremstyle{definition}

\theoremstyle{remark}
\newtheorem{rmk}[thm]{Remark}
\newtheorem{ex}[thm]{Example}

\newcommand{\C}{\mathbb{C}}
\newcommand{\N}{\mathbb{N}}

\newcommand{\T}{\mathbb{T}}
\newcommand{\Z}{\mathbb{Z}}

\newcommand{\LL}{\mathcal{L}}
\newcommand{\go}{G^{(0)}}

\newcommand{\calK}{\mathcal{K}}
\newcommand{\red}{\operatorname{r}}
\newcommand{\supp}{\operatorname{supp}}
\newcommand{\Ind}{\operatorname{Ind}}
\newcommand{\dashind}{\operatorname{\!-Ind}}

\newcommand{\tr}{\operatorname{tr}}
\newcommand{\inter}{\operatorname{int}}

\newcommand{{\Kum}}{\operatorname{Kum}}
\newcommand{{\MW}}{\operatorname{\frak{MW}}}

\newcommand{\deltaDD}{\delta_{\operatorname{DD}}}

\begin{document}
\title[Fell algebras]{The equivalence relations of local homeomorphisms\\
and Fell algebras}
\author[Clark]{Lisa Orloff Clark}
\author[an Huef]{Astrid an Huef}
\author[Raeburn]{Iain Raeburn}
\address{Department of Mathematics and Statistics, University of Otago, PO Box 56, Dunedin 9054, New Zealand.}
\email{\{lclark, astrid, iraeburn\}@maths.otago.ac.nz}

\date{\today}

\begin{abstract}  
We study the groupoid $C^*$-algebras $C^*(R(\psi))$ associated to the equivalence relation $R(\psi)$ induced by a quotient map $\psi:Y\to X$. If $Y$ is Hausdorff then $C^*(R(\psi))$ is a Fell algebra, and if both $Y$ and $X$ are Hausdorff then $C^*(R(\psi))$ has continuous trace.   We show that the $C^*$-algebra $C^*(G)$ of a locally compact, Hausdorff and principal groupoid $G$ is a Fell algebra if and only if $G$ is topologically isomorphic to some $R(\psi)$, extending a theorem of Archbold and Somerset. 
The algebras $C^*(R(\psi))$ are, up to Morita equivalence, precisely the Fell algebras with trivial Dixmier-Douady invariant as recently defined by an Huef, Kumjian and Sims. We use a twisted analogue of $C^*(R(\psi))$ to provide examples of Fell algebras with non-trivial Dixmier-Douady invariant.

 \end{abstract}

\subjclass[2000]{46L55}

\keywords{Fell algebra; continuous-trace algebra; Dixmier-Douady invariant; the $C^*$-algebra of a local homeomorphism; groupoid $C^*$-algebra}
\maketitle

\section{Introduction}
Important classes of type I $C^*$-algebras include the continuous-trace
algebras and the liminary or CCR algebras, and lately there has been
renewed interest in a family of liminary algebras called Fell algebras
\cite{aH,C,AaH,aHKS}. A $C^*$-algebra is a \emph{Fell
algebra} if for each $\pi\in \hat A$, there exists  a positive element $a$ such that
$\rho(a)$ is a rank-one projection for all $\rho$ in a neighbourhood of
$\pi$ in $\hat A$. Fell algebras were named by Archbold and Somerset \cite{AS} in respect of Fell's contributions \cite{F}, and had been previously studied by Pedersen as ``algebras of type I$_0$'' \cite[\S6]{Ped}. Proposition~4.5.4 of \cite{D} says that a
$C^*$-algebra has continuous trace if and only if it is a Fell algebra and
its spectrum is Hausdorff.

The Dixmier-Douady class ${\deltaDD}(A)$ of a continuous-trace algebra $A$
identifies $A$ up to Morita equivalence, and $\deltaDD(A)=0$ if and only
if $A$ is Morita equivalent to a commutative $C^*$-algebra
\cite[Theorem~5.29]{tfb}. An Huef, Kumjian and Sims have recently
developed an analogue of the Dixmier-Douady classification for Fell
algebras \cite{aHKS}. Their Dixmier-Douady invariant $\delta(A)$ vanishes
if and only if $A$ is Morita equivalent to the groupoid $C^*$-algebra
$C^*(R(\psi))$ of the equivalence relation associated to a local
homeomorphism $\psi$ of a Hausdorff space  onto the spectrum $\hat A$. (Since this theorem
was not explicitly stated in \cite{aHKS}, we prove it here as
Theorem~\ref{thm-true}.)

So the theory in \cite{aHKS} identifies the $C^*$-algebras $C^*(R(\psi))$
as an interesting family of model algebras. Here we investigate the
structure of the algebras $C^*(R(\psi))$ and their twisted analogues, and use
them to provide  examples of Fell algebras exhibiting certain kinds of
behaviour. In particular, we will produce some concrete examples of Fell
algebras with non-vanishing Dixmier-Douady invariant. 

After some background in \S\ref{sec-back},  we discuss in \S\ref{sec-top} the topological spaces that arise as the spectra of Fell algebras.
In \S\ref{sec:C*Rpsi} we study the locally compact Hausdorff
equivalence relation $R(\psi)$ associated to a surjection $\psi:Y\to X$
defined on a locally compact Hausdorff space $Y$. We show how extra
properties of $\psi$ influence the structure of the groupoid $R(\psi)$. We
show in particular that if $\psi$ is a surjective local homeomorphism of
$Y$ onto a topological space $X$, then $R(\psi)$ is \'etale, principal and
Cartan, with orbit space naturally homeomorphic to $X$. We also show that every principal Cartan groupoid has the form $R(\psi)$ for some quotient map $\psi$.

In \S\ref{Lisa}, we show that the twisted groupoid $C^*$-algebras of the $R(\psi)$ give many examples of Fell algebras. We also show, extending a result of  Archbold and Somerset for \'etale groupoids \cite{AS}, that the $C^*$-algebra $C^*(G)$ of a principal groupoid $G$ is a Fell algebra if and only if $G$ is topologically isomorphic to the relation $R(q)$ determined by the quotient map $q$ of the unit space $\go$ onto $\go/G$. We then illustrate our results with a discussion of the path groupoids of directed graphs, and an example from \cite{HaH} which fails to be Fell in a particularly delicate way. 

In Theorem~\ref{thm-true}, we prove that the Dixmier-Douady class
$\delta(A)$ of a Fell algebra $A$ vanishes if and only if $A$ is Morita
equivalent to some $C^*(R(\psi))$. We then use this to partially resolve a
problem left open in \cite[Remark~7.10]{aHKS}: when $A$ has continuous trace, how is $\delta(A)$ related to the usual Dixmier-Douady invariant $\deltaDD(A)$ of \cite{DD,D,tfb}? In Corollary~\ref{cor:0iff0},
we show that when $A$ has continuous trace, $\delta(A)=0$ if and only if $\deltaDD(A)=0$. We also show that if $A$ is Fell and $\delta(A)=0$, then every ideal $I$ in $A$ with continuous trace has $\deltaDD(I)=0$. Since $\deltaDD$ is
computable, this allows us to recognise some Fell algebras whose invariant
is nonzero.

In \S\ref{examples}, we describe two examples of Fell algebras which we
have found instructive. The first is a Fell algebra whose spectrum fails
to be paracompact in any reasonable sense, even though the algebra is
separable. The second set of examples are Fell algebras $A$ with nonzero invariant $\delta(A)$ and non-Hausdorff spectrum (so that they are not continuous-trace algebras). We close \S\ref{examples} with a brief
epilogue on how we found these examples and what we have learned from
them.

We finish with two short appendices. The first concerns the different
twisted groupoid algebras appearing in this paper. We mainly use Renault's algebras associated to a $2$-cocycle $\sigma$ on $G$ from \cite{Ren}, but the proof of Theorem~\ref{thm-true} uses
the twisted groupoid algebra associated to a twist $\Gamma$ over
$G$ from \cite{K}, and the proof of Theorem~\ref{Cartan=Fell} uses yet another version from \cite{MW}. In Appendix~\ref{app-twists2}, we show
that when $\Gamma$ is the twist associated to a continuous cocycle, the
three reduced $C^*$-algebras are isomorphic. 
In the last appendix, we describe the continuous-trace ideal in an arbitrary $C^*$-algebra. In the end, we did not need this result, but we think it may be of some general interest: we found it curious that the ideas which work for
transformation group algebras in \cite[Corollary~18]{green} and \cite[Theorem~3.10]{aHW} work equally well in arbitrary $C^*$-algebras.

\section{Notation and background}\label{sec-back}

 A \emph{groupoid} $G$ is a small category in which every morphism is invertible.  We write $s$ and $r$ for the
domain  and range maps in $G$.  The set $\go$ of objects in $G$ is called the 
\emph{unit space}, and we frequently identify a unit with the identity morphism at that unit. A groupoid is  \emph{principal} if there is at most one morphism between each pair of units. 

A \emph{topological groupoid} is a groupoid equipped with a topology on the set of morphisms such that the composition and inverse maps are continuous.   A topological groupoid $G$ is \emph{\'etale} if 
the map $r$ (equivalently, $s$) is a local homeomorphism.  
The unit space of an \'etale groupoid is open in $G$, and  the sets $s^{-1}(u)$ and $r^{-1}(u)$ are discrete for every $u\in \go$.

Suppose $G$ is a topological groupoid. Then the \emph{orbit} of $u\in  G^{(0)}$ is $[u]:=r(s^{-1}(u))$.  For $u,v\in \go$ we write $u\sim v$ if $[u]=[v]$, and then $\sim$ is an equivalence relation on $\go$.   We write $q:G^{(0)}\to \go/G:=\go/\!\!\sim$ for the quotient map onto the orbit space.  If $G$ is \'etale, then $r$ is open, and then the quotient map is also open because  $q^{-1}(q(U))=r(s^{-1}(U))$ for $U\subset \go$.

A topological groupoid $G$ is \emph{Cartan} if every unit $u \in \go$ has a neighbourhood $N$ in $\go$ which is \emph{wandering} in the sense that $s^{-1}(N)\cap r^{-1}(N)$ has compact closure.

Let $G$ be a locally compact Hausdorff  groupoid with left Haar system $\lambda=\{\lambda^u:u\in G^{(0)}\}$. We also need to use the corresponding right Haar system $\{\lambda_u\}$ defined by $\lambda_u(E)=\lambda^u(E^{-1})$. A \emph{$2$-cocycle} on $G$ is a function $\sigma:G^{(2)}\to\T$ such that $\sigma(\alpha,\beta)\sigma(\alpha\beta,\gamma)=\sigma(\beta,\gamma)\sigma(\alpha,\beta\gamma)$. As in \cite{BaH}, we assume that all our cocycles are continuous and normalised in the sense that $\sigma(r(\gamma),\gamma)=1=\sigma(\gamma,s(\gamma))$, and we write $Z^2(G,\T)$ for the set of such cocycles. For such $\sigma$, there are both full and reduced twisted groupoid $C^*$-algebras. Here we work primarily with the
reduced version, though in fact the full and reduced groupoid $C^*$-algebras coincide for the groupoids of interest to us (see Theorem~\ref{Cartan=Fell}). Let $C_c(G, \sigma)$ be $C_c(G)$ with involution   and convolution  given by $f^*(\alpha) = \overline{f(\alpha^{-1})}\sigma(\alpha,\alpha^{-1})$ and 
\begin{align*}
    (f*g)(\alpha) &= \int_G f(\alpha\gamma) g(\gamma^{-1})\sigma(\alpha\gamma,\gamma^{-1})\, d\lambda^{s(\alpha)}(\gamma);
\end{align*}
it is shown in \cite[Proposition~II.1.1]{Ren} that $C_c(G, \sigma)$ is a $*$-algebra.  The invariance of the Haar system gives
\begin{equation}\label{eq-ops2}
(f*g)(\alpha)= \int_G f(\beta) g(\beta^{-1}\alpha)\sigma(\beta,\beta^{-1}\alpha)\, d\lambda^{r(\alpha)}(\beta).
\end{equation}
For $u \in \go$, there is an \emph{induced representation} $\Ind_u^\sigma$  of $C_c(G,\sigma)$
on $L^2(s^{-1}(u),\lambda_u)$ such that, for $f\in C_c(G,\sigma)$ and $\xi\in L^2(s^{-1}(u),\lambda_u)$, 
\[(\Ind_u^\sigma(f)\xi)(\alpha)
= \int_G f(\beta) \xi(\beta^{-1}\alpha)\sigma(\beta,\beta^{-1}\alpha)\, d\lambda^{r(\alpha)}(\beta).\]
As in \cite[\S II.2]{Ren}, the \emph{reduced twisted groupoid $C^*$-algebra} $C^*_{\red}(G,\sigma)$ is the completion of $C_c(G,\sigma)$ with respect to the reduced norm
\[\|f\|_{\red} = \sup_{u \in\go} \|\Ind_u^\sigma(f)\|.\]
As usual, we write $C^*_{\red}(G)$ for $C^*_{\red}(G,1)$ and $\Ind_u$ for $\Ind_u^1$.

If $G$ is \'etale then $r^{-1}(r(\beta))$ is discrete, and \eqref{eq-ops2}, for example, reduces to  \[
f*g(\alpha)=\sum_{r(\alpha) =r(\beta)} f(\beta) g(\beta^{-1}\alpha)\sigma(\beta,\beta^{-1}\alpha)=\sum_{\alpha =\beta \gamma} f(\beta) g(\gamma)\sigma(\beta,\gamma).
\]

\section{Topological preliminaries}\label{sec-top}

By the standard definition, a topological space $X$ is \emph{locally compact} if every 
point of $X$ has a compact neighbourhood. When $X$ is Hausdorff, this is equivalent to asking that every point has a neighbourhood base of compact sets \cite[Theorem~29.2]{M}. For general, possibly non-Hausdorff spaces, we say that $X$ is \emph{locally locally-compact} if every point of 
$X$ has a neighbourhood basis of compact sets (Lemma~\ref{Lem:LLC_Equiv} below explains our choice of name). The spectrum of a $C^*$-algebra is always locally locally-compact (see Lemma~\ref{Lem:FellLLC_LH}), so this ``neighbourhood basis'' version of local compactness has attractions for operator algebraists. It also has the advantage, as Munkres points out in \cite[page~185]{M}, that it is more consistent with other uses of the word ``local'' in topology. It has been adopted without comment as \emph{the} definition of local compactness in \cite[page~149, problem~29]{B}\footnote{Modulo Bourbaki's use of the word ``quasi-compact'' to mean what we call compact. Dixmier follows Bourbaki, as one should be aware when reading \cite{D}.} and in \cite[Definition~1.16]{W}. However, many topology books, such as \cite{Kelley} and \cite{M}, and many real-analysis texts, such as \cite{Rud} and \cite{P}, use the standard ``every point has a compact neighbourhood'' definition, and we will go along with them.

The proof of the following lemma is straightforward.

\begin{lemma}\label{Lem:LLC_Equiv}
A topological space $X$ is locally locally-compact if and only if every open subset of $X$ is locally compact.
\end{lemma}

A topological space $X$ is \emph{locally Hausdorff} if every point of $X$ has a 
Hausdorff neighbourhood. It is straightforward to verify that a locally Hausdorff space is T${}_1$. The following result explains our interest in locally locally-compact and locally Hausdorff spaces. 


\begin{lemma}\label{Lem:FellLLC_LH}
If $A$ is a Fell algebra, then the spectrum of $A$ is locally locally-compact and locally Hausdorff.
\end{lemma}

\begin{proof}
Corollary~3.3.8 of \cite{D} implies that $\hat A$ is locally locally-compact, and Corollary 3.4 of \cite{AS} that $\hat{A}$ is locally Hausdorff.
\end{proof}

Lemma~\ref{Lem:FellLLC_LH} has a converse: every locally locally-compact and locally Hausdorff space is the spectrum of some Fell algebra \cite[Theorem~6.6(2)]{aHKS}. We will later give a shorter proof of this result (see Corollary~\ref{new}).  

We warn that a locally locally-compact and locally Hausdorff space may not be paracompact, that compact subsets may not be closed, and that the intersection of two compact sets may not be compact (for example, in the spectrum of the Fell $C^*$-algebra described in \S\ref{sec-example}). So we have found our usual, Hausdorff-based intuition to be distressingly misleading, and we have tried to exercise extreme caution in matters topological.

Locally locally-compact and locally Hausdorff spaces have the following purely topological characterisation.

\begin{prop}\label{charllc}
\begin{enumerate}
\item\label{charllc-a}
 Let $\psi:Y\to X$ be a local homeomorphism of a   
locally compact Hausdorff space $Y$ onto  a topological space  $X$. Then $X$ is locally locally-compact and  locally Hausdorff.  If $Y$ is second-countable, so is $X$. 

\item\label{charllc-b} Let $X$ be a locally locally-compact and  locally Hausdorff space. Then there are a locally compact Hausdorff space $Y$ and a local homeomorphism $\psi$ of $Y$ onto $X$. If $X$ is second-countable, then we can take $Y$ to be second-countable.
\end{enumerate}
\end{prop}

\begin{proof}
For \eqref{charllc-a} suppose that $\psi:Y\to X$ is a local homeomorphism. Fix $x\in X$ and an open neighbourhood $W$ of $x$ in $X$.  Let $y\in\psi^{-1}(x)$. Since $\psi^{-1}(W)$ is an open neighbourhood of $y$ and $\psi$ is a local 
homeomorphism, there is a neighbourhood $U$ of $y$ contained in $\psi^{-1}(W)$ such that $\psi|_U$ is a homeomorphism. Since $Y$ is locally compact and  Hausdorff, it is locally locally-compact by \cite[Theorem~29.2]{M}, and there is  a compact  neighbourhood $K$ of $y$ contained in $U$.  
Then  $\psi(K)$ is compact and Hausdorff, and because $\psi$ is open, it is a neighbourhood of $x$ contained in $W$.  This proves both that $X$ is locally locally-compact and that $X$ is locally Hausdorff.

Since $\psi$ is continuous and open, the image of a basis for the topology on $Y$ is a basis for the topology on $X$. Thus $X$ is second-countable if $Y$ is.

For \eqref{charllc-b}, suppose that $X$ is locally locally-compact and locally Hausdorff. Choose an open cover $\mathcal{U}$ of $X$  by Hausdorff sets.  Let $Y:= \bigsqcup_{U\in \mathcal{U}} U$, and topologise $Y$ by giving each $U$ the subspace topology from $X$ and making each $U$ open and closed in $Y$.
Then Lemma~\ref{Lem:LLC_Equiv} implies that $Y$ is locally compact and Hausdorff, and the  inclusion maps $U\to X$ combine to give a surjective local homeomorphism $\psi:Y \to X$. If $X$ is second-countable, then we can take the cover to be countable, and $Y$ is also second-countable.
\end{proof}


\section{The groupoid associated to a local homeomorphism}\label{sec:C*Rpsi}

Let $\psi$ be  a surjective map from a topological space $Y$ to a set $X$, and take
\[
R(\psi)=Y\times_\psi Y:= \{(y,z) \in Y \times Y : \psi(y) = \psi(z)\}.
\]
With the subspace topology and the operations $r(y,z)=y$, $s(y,z)=z$ and $(x,y)(y,z)=(x,z)$, $R(\psi)$ is a principal topological groupoid with unit space $Y$.   

We want to examine the effect of properties of $Y$ and $\psi$ on the structure of $R(\psi)$. We begin by looking at the orbit space.

\begin{lemma}\label{X=orbitsp} Let $\psi$ be  a surjective map from a topological space $Y$ to a set $X$, and define $h:X\to Y/R(\psi)$ by $h(x)=\psi^{-1}(x)$. 
\begin{enumerate}
\item\label{itema} The function $h$ is a bijection, and $h\circ \psi$ is the quotient map $q:Y\to Y/R(\psi)$.
\item\label{itemb} If $X$ is a topological space and $\psi$ is continuous, then $h$ is open.
\item\label{itemc} Suppose that $\psi:Y\to X$ is a quotient map, in the sense that $U$ is open in $X$ if and only if $\psi^{-1}(U)$ is open. Then $h$ is a homeomorphism of $X$ onto $Y/R(\psi)$.
\end{enumerate}
\end{lemma}

\begin{proof} \eqref{itema} If $h(x)=h(x')$, then the surjectivity of $\psi$ implies that there exists at least one $z\in \psi^{-1}(x)=\psi^{-1}(x')$, and then $x=\psi(z)=x'$. So $h$ is one-to-one. Surjectivity is easy: every orbit $\psi^{-1}(\psi(y))=h(\psi(y))$. The same formula $h(\psi(y))=\psi^{-1}(\psi(y))$ shows that $h\circ\psi(y)$ is the orbit $q(y)$ of $y$.

For \eqref{itemb}, take $U$ open in $X$. Then $q^{-1}(h(U))=\psi^{-1}(h^{-1}(h(U)))=\psi^{-1}(U)$ is open because $\psi$ is continuous, and then $h(U)$ is open by definition of the quotient topology. For \eqref{itemc}, take $V$ open in $Y/R(\psi)$. Then $\psi^{-1}(h^{-1}(V))=q^{-1}(V)$ is open in $Y$, and $h^{-1}(V)$ is open because $\psi$ is a quotient map.
\end{proof} 

\begin{lemma}\label{lem:qtnt map}
Suppose that $\psi:Y \to X$ is a quotient map.   
Then $R(\psi)$ is \'etale if and only if $\psi$ is a local homeomorphism.
\end{lemma}

\begin{proof}
Suppose that $\psi$ is a local homeomorphism and $(y,z)\in R(\psi)$. There are open neighbourhoods $U$ of $y$ and $V$ of $z$ such that $\psi|_U$ and $\psi|_V$ are homeomorphisms onto open neighbourhoods of $\psi(y)=\psi(z)$. By shrinking if necessary, we may suppose $\psi(U)=\psi(V)$. Now $W:=(U\times V)\cap R(\psi)$ is an open neighbourhood of $(y,z)$, and the function $w\mapsto (w,(\psi|_V)^{-1}\circ\psi|_U(w))$ is a continuous inverse for $r|_W$. Thus $r$ is a local homeomorphism, and $R(\psi)$ is \'etale.

Conversely, suppose that $R(\psi)$ is \'etale and $y\in Y$. Then $(y,y)\in R(\psi)$. Since $r$ is a local homeomorphism, there is a neighbourhood $W$ of $(y,y)$ in $R(\psi)$ such that $r|_W$ is a homeomorphism. By shrinking, we can assume that $W=(U\times U)\cap R(\psi)$ for some open neighbourhood $U$ of $y$ in $Y$. We claim that $\psi$ is one-to-one on $U$. Suppose $y_1,y_2\in U$, and $\psi(y_1)=\psi(y_2)$. Then $(y_1,y_1)$ and $(y_1,y_2)$ are both in $W$.
Now $r(y_1,y_1)=y_1=r(y_1,y_2)$, the injectivity of $r|_W$ implies that $(y_1,y_1)=(y_1,y_2)$, and $y_1=y_2$. Thus $\psi|_U$ is one-to-one, as claimed. Since the orbit map $q$ in an \'etale groupoid is open, and $h$ is a homeomorphism with $h\circ\psi=q$, it follows that $\psi(U)=h^{-1}(q(U))$ is open. Thus $\psi$ is a local homeomorphism.
\end{proof}

\begin{prop}
\label{prop:r(psi)_props}
Suppose that $Y$ and $X$ are topological spaces with $Y$ locally compact Hausdorff, and $\psi:Y \to X$ is a surjective local homeomorphism.
Then $R(\psi)$ is locally compact, Hausdorff,  principal, \'etale and Cartan.
\end{prop}

\begin{proof}
The groupoid $R(\psi)$ is principal because it is an equivalence relation, and is Hausdorff because $Y$ is Hausdorff.
Since $\psi$ is a local homeomorphism, it is a quotient map, and hence Lemma~\ref{lem:qtnt map} implies that $R(\psi)$ is  \'etale. Let $(y,z)\in R(\psi)$. Then there is an open neighbourhood $U$ of $(y,z)$ such that $r|_U$ is a homeomorphism onto an open neighbourhood of $y$. Since locally compact Hausdorff spaces are locally locally-compact, $r(U)$ contains a compact neighbourhood $K$ of $y$. Now $(r|_U)^{-1}(K)$ is a compact neighbourhood of $(y,z)$ in $R(\psi)$, and we have shown that $R(\psi)$ is locally compact.

Since locally compact Hausdorff spaces are regular \cite[1.7.9]{P}, the following lemma tells us that $R(\psi)$ is Cartan, and hence completes the proof of Proposition~\ref{prop:r(psi)_props}. The extra generality in the lemma will be useful in the proof of Proposition~\ref{prop:cartan_iff_pi_homeo}.
\end{proof}
    
\begin{lemma}\label{lem:lc_implies_cartan}
Suppose that $Y$ is a regular topological space and $\psi:Y\to X$ is a surjection. If $R(\psi)$ is locally compact, then $R(\psi)$ is Cartan. 
 \end{lemma}
 
\begin{proof} 
Let $y\in Y$. We must find a wandering neighbourhood $W$ of $y$ in $Y$, that is, a neighbourhood $W$ such that $(W\times W)\cap R(\psi)$ has compact closure in $R(\psi)$.  
Since $R(\psi)$ is locally compact, there exists a compact neighbourhood $K$ of $(y,y)$ in $R(\psi)$. Since $K$ is a neighbourhood, the interior $\inter K$ is an open set containing $(y,y)$, and there exists an open
set $O \subset Y \times Y$ such that  $\inter K = O \cap R(\psi)$.

Choose open neighbourhoods  $U_1$, $U_2$ of $y$ in $Y$ such that $U_1 \times U_2 \subset O$.  Since $Y$ is 
regular, there are open neighbourhoods $V_i$ of $y$ such that 
$\overline{V}_i \subset U_i$ for  $i=1,2$.
Let $C := \overline{V}_1 \cap \overline{V}_2$.  Then $C$ is a closed neighbourhood of $y$ in $Y$, and
\[
(C \times C) \cap R(\psi) \subset (U_1 \times U_2) \cap R(\psi) \subset O \cap R(\psi) \subset K.\]
Since $(C \times C) \cap R(\psi)$ is closed in $R(\psi)$ and $K$ is compact, 
$(C \times C) \cap R(\psi)$ is compact, and $C$ is the required neighbourhood of $y$.
\end{proof}

Proposition~\ref{prop:r(psi)_props} has an intriguing converse. Suppose that $G$ is a locally compact, Hausdorff and principal groupoid. We will see that if $G$ is Cartan, then $G$ has the form $R(q)$, where $q:\go\to \go/G$ is the quotient map. The key idea is that, because $G$ is principal, the map $r\times s:\gamma\mapsto (r(\gamma),s(\gamma))$ is a groupoid isomorphism of $G$ onto $R(q)$. The map $r\times s$ is also continuous for the product topology on $R(q)$, but it is not necessarily open, and hence is not necessarily an isomorphism of topological groupoids (see Example~\ref{HaHgraph} below). But:

\begin{prop}\label{prop:cartan_iff_pi_homeo}
Suppose $G$ is a locally compact, Hausdorff and principal groupoid which admits a Haar system. Then $G$ is Cartan if and only if $r\times s$ is a topological isomorphism of $G$ onto $R(q)$.
\end{prop}

For the proof we need a technical lemma.

\begin{lemma}
\label{prop:Cartan_implies_pi_homeo}
Let  $G$ be a locally compact Hausdorff groupoid which admits a Haar system. If $G$ is Cartan, then $r\times s$ is a closed map onto its image $r\times s(G)$. 
\end{lemma}

\begin{proof}
Let $C$ be a closed subset of $G$ and $(u,v)$ be a limit point of $r\times s(C)$ in $r\times s(G)$.  Then there exists $\gamma \in G$ such that $r\times s(\gamma) = (u,v)$ and a net $\{\gamma_i\}$ in $C$ such that $r\times s(\gamma_i) \to (u,v)$.  It suffices to show that 
$\{\gamma_i\}$ has a convergent subnet. Indeed, if $\gamma_{i_j} \to \gamma'$ then   $\gamma' \in C$ because $C$ is closed,  
$r\times s(\gamma')=(u,v)$ by continuity, and $(u,v) \in r\times s(C)$.

Since $G$ is Cartan, $u$ has a neighbourhood  $N$ in $G^{(0)}$ such that $U:=(r\times s)^{-1}(N \times N)$ is relatively compact in $G$.  Let  $V$ be a relatively compact neighbourhood of $\gamma$.  We may assume by shrinking $V$ that $V\subset r^{-1}(N)$, and hence that $r(V) \subset N$.   
The continuity of multiplication implies that $UV$ is relatively compact. 

We claim that $\gamma_i \in UV$ eventually.  To see this, we observe that the existence of the Haar system implies that $s$ is open \cite[Corollary, page~118]{S}, and hence $s(V)$ is a neighbourhood of $v=s(\gamma)$. Thus there exists $i_0$ such that  $s(\gamma_i) \in s(V)$ and $r(\gamma_i) \in N$ for all $i\geq i_0$.  For each $i\geq i_0$ there exists $\beta_i \in V$ such that $s(\gamma_i) = s(\beta_i)$.  Now
$s(\gamma_i\beta_i^{-1}) = r(\beta_i) \in r(V)\subset N$ and 
$r(\gamma_i\beta_i^{-1}) = r(\gamma_i) \in  N$, so  $\alpha_i:=\gamma_i\beta_i^{-1}$ is in $U$.   Thus  $\gamma_i = \alpha_i\beta_i \in UV$ for $i\geq i_0$, as claimed. Now $\{\gamma_i:i\geq i_0\}$ is a net in a 
relatively compact set, and hence has a convergent subnet, as required.
\end{proof}

\begin{proof}[Proof of Proposition~\ref{prop:cartan_iff_pi_homeo}]
Suppose that $G$ is Cartan.  The map $r\times s$ is always a continuous surjection onto $R(q)$, and it is injective because $G$ is principal. Since $G$ is Cartan, Lemma~\ref{prop:Cartan_implies_pi_homeo}  implies that $r\times s$ is closed as a map onto its image $R(q)$.  A bijection is open if and only if it is closed, so $r\times s$ is open. Hence $r\times s$ is a homeomorphism onto $R(q)$.

Conversely, suppose that $r\times s$ is a homeomorphism onto $R(q)$. Then $r\times s$ is an isomorphism of topological groupoids, and since $G$ is locally compact, so is $R(q)$. Thus $R(q)$ is Cartan by Lemma~\ref{lem:lc_implies_cartan}, and so is $G$.
\end{proof}

\subsection*{Concluding discussion} Lemma~\ref{lem:lc_implies_cartan} shows that, if the groupoid $R(\psi)$ associated to a quotient map is locally compact, then $R(\psi)$ is Cartan. On the other hand, Proposition~\ref{prop:cartan_iff_pi_homeo} says that, if there is a topology on $R(\psi)$ which makes it into a locally compact Cartan groupoid, then that topology has to be the relative topology from the product space $Y\times Y$. So one is tempted to seek conditions on $\psi$ which ensure that the subset $R(\psi) \subset Y\times Y$ is locally compact. By Proposition~\ref{prop:r(psi)_props}, it suffices for $\psi$ to be a local homeomorphism. This is not a necessary condition: for example, if $(Y,H)$ is a free Cartan transformation group with $H$ nondiscrete, then the transformation groupoid  $Y\times H$ is Cartan in our sense, but the quotient map $q:Y\to Y/H$ is not locally injective. (In \cite{aH} there is a specific example of a free Cartan transformation group which illustrates this.) However, the following example shows that something extra is needed.

\begin{ex}
Take $Y=[0,1]$ and $X=\{a,b\}$ with the topology $\{X,\{a\},\emptyset\}$, and define $\psi:Y\to X$ by 
\[
\psi(y)=\begin{cases}a&\text{if $t>0$}\\
b&\text{if $t=0$.}
\end{cases}
\]
Then $\psi$ is an open quotient map, but $R(\psi)=\big((0,1]\times (0,1]\big)\cup \{(0,0)\}$ is not a locally compact subset of $[0,1]\times[0,1]$ because $(0,0)$ does not have a compact neighbourhood.
\end{ex}

\section{The groupoids whose $C^*$-algebras are Fell algebras}\label{Lisa}

We begin this section by summarising results from \cite{C,CaH2,BaH,MW2} about the twisted groupoid algebras of principal Cartan groupoids.

\begin{thm}\label{Cartan=Fell}
Suppose that $G$ is a second-countable, locally compact, Hausdorff, principal and Cartan groupoid which admits a Haar system. Then the following statements are equivalent:
\begin{enumerate}
\item\label{C=Fa} $G$ is Cartan;
\item\label{C=Fb} $C^*(G)$ is a Fell algebra;
\item\label{C=Fd} for all $\sigma\in Z^2(G,\T)$, $C^*(G, \sigma)$ is a Fell algebra.
\end{enumerate} 
If \eqref{C=Fa}--\eqref{C=Fd} these are satisfied, then $u\mapsto \Ind_u^{{\sigma}}$ induces a homeomorphism of $\go/G$ onto $C^*(G, \sigma)^\wedge$, and $C^*(G,\sigma)=C^*_{\red}(G,\sigma)$.
\end{thm}

\begin{proof}
The implication $\eqref{C=Fd}\Longrightarrow
\eqref{C=Fb}$ is trivial, the equivalence of \eqref{C=Fa} and \eqref{C=Fb} is Theorem~7.9 of \cite{C}, and the implication $\eqref{C=Fb}\Longrightarrow\eqref{C=Fd}$ is part (ii) of \cite[Proposition~3.10\,(a)]{BaH}. So it remains to prove the assertions in the last sentence. 

The orbits in a Cartan groupoid are closed \cite[Lemma~7.4]{C}, so the orbit space $Y/R(\psi)$ is T${}_1$, and it follows from \cite[Proposition~3.2]{CaH2} that the map $y\mapsto[L^u]$ described there induces a homeomorphism of $\go/G$ onto the spectrum of the twisted groupoid $C^*$-algebra $C^*(G^{\overline{\sigma}}; G)^{\MW}$ of Muhly and Williams (see \cite{MW2} or Appendix~\ref{app-twists2}). By \cite[Lemma~3.1]{BaH}, $C^*(G^{\overline{\sigma}}; G)^{\MW}$  is isomorphic  to $C^*(G,{\sigma})$, and by Lemma~\ref{lem-rho},  this isomorphism  carries the equivalence class of  $L^u$ to  the class of $\Ind^{{\sigma}}_u$. We deduce that $u\mapsto \Ind^{{\sigma}}_u$ induces a homeomorphism, as claimed. This implies in particular that all the  irreducible representations of $C^*(G, \sigma)$ are induced, so for $f\in C_c(G, {\sigma})$ we have
\[
\|f\|=\sup\{\|\Ind_y(f)\|:y\in Y\}=:\|f\|_{\red},
\]
and $C^*(G,{\sigma})=C^*_{\red}(G,{\sigma})$.
\end{proof}

For our first application of Theorem~\ref{Cartan=Fell}, we observe that putting the equivalence of \eqref{C=Fa} and \eqref{C=Fb} together with Proposition~\ref{prop:cartan_iff_pi_homeo} gives the following improvement of a result of Archbold and Somerset \cite[Corollary~5.9]{AS}. (We discuss the precise connection with \cite{AS} in Remark~\ref{rmk-AS}.)

\begin{cor}
\label{cor:Fell_iff_pi_homeo}
Suppose that $G$ is a second-countable, locally compact, Hausdorff and principal groupoid which admits a Haar system $\lambda$, and $q:\go\to \go/G$ is the quotient map.  Then $C^*(G,\lambda)$ is a Fell algebra if and only if $r\times s:G\to\go\times\go$ is a topological isomorphism of $G$ onto $R(q)$.
\end{cor}

\begin{rmk}\label{rmk-AS} 
The ``separated topological equivalence relations'' $R$ studied in \cite[\S5]{AS} are the second-countable, locally compact, Hausdorff and principal groupoids that are  \'etale.  When they say in \cite[Corollary~5.9]{AS} that ``the topologies $\tau_p$ and $\tau_0$ coincide,'' they mean precisely that the map $r\times s$ is a homeomorphism for the original topology $\tau_0$ on $R$ and the product topology on $R^{(0)}\times R^{(0)}$. Theorem~\ref{Cartan=Fell} implies that the full algebra above and the reduced algebra in \cite{AS} coincide. 
So Corollary~\ref{cor:Fell_iff_pi_homeo} extends \cite[Corollary~5.9]{AS} from principal \`etale groupoids to principal groupoids which admit a Haar system. This is a substantial generalisation since, for example, locally compact transformation groups always admit a Haar system \cite[page 17]{Ren} even though the associated transformation groupoids may not be \'etale. Our proof of Corollary~\ref{cor:Fell_iff_pi_homeo} seems quite different from the representation-theoretic arguments used in \cite{AS}. \end{rmk}

Next we apply Theorem~\ref{Cartan=Fell} to the groupoid associated to a local homeomorphism.

\begin{cor}\label{thm standard groupoid}
Let $\psi:Y\to X$ be a local homeomorphism of a second-countable,  
locally compact and Hausdorff space $Y$ onto  a topological space  $X$, and let $\sigma:R(\psi)^{(2)}\to\T$ be a continuous normalised $2$-cocycle. Then $C^*(R(\psi), \sigma)$ is a Fell algebra, $y\mapsto\Ind_y^{{\sigma}}$ 
induces a homeomorphism of $Y/R(\psi)$ onto $C^*(R(\psi), \sigma)^\wedge$, and
$C^*(R(\psi),\sigma)=C^*_{\red}(R(\psi),\sigma)$.
\end{cor}

\begin{proof} 
Proposition~\ref{prop:r(psi)_props} implies that $R(\psi)$ satisfies all the hypotheses of Theorem~\ref{Cartan=Fell}, which then gives the result.
\end{proof}

When $X$ is Hausdorff, we know from \cite{Kumjian} that $C^*(R(\psi))$ has cotinuous trace; the twisted versions we used in \cite{RT} to provide examples of continuous-trace algebras with nonzero Dixmier-Douady class.

We now give our promised shorter proof of the converse of Lemma~\ref{Lem:FellLLC_LH}.

\begin{cor}\label{new}
Let $X$ be a second-countable, locally locally-compact and locally Hausdorff topological space. Then there is a separable  Fell $C^*$-algebra with spectrum $X$.
\end{cor}

\begin{proof}
By Proposition~\ref{charllc}, there are a second-countable, locally compact and  Hausdorff space $Y$ and a surjective local homeomorphism $\psi:Y\to X$. Consider the topological relation $R(\psi)$. 
Lemma~\ref{X=orbitsp} gives a homeomorphism $h$ of $X$ onto $Y/R(\psi)$, and Corollary~\ref{thm standard groupoid} says that $C^*(R(\psi))$ is a separable Fell algebra with spectrum homeomorphic to $Y/R(\psi)$.
\end{proof}

Next we consider a row-finite directed graph $E$ with no sources, using the conventions of \cite{CBMS}; we also use the more recent convention that, for example, 
\[
vE^nw=\{\alpha\in E^n:r(\alpha)=v,\;s(\alpha)=w\}. 
\]
The infinite-path space $E^\infty$ has a locally compact Hausdorff topology with basis the cylinder sets
\[
Z(\alpha)=\{\alpha x:\text{ $x\in E^\infty$ and $r(x)=s(\alpha)$}\}.
\]
\cite[Corollary~2.2]{KPRR}. The set
\begin{equation*}\label{defGE}
G_E=\{(x,k,y)\in E^\infty\times \Z\times E^\infty:\text{there exists $n$ such that $x_i=y_{i+k}$ for $i\geq n$}\}
\end{equation*}
is a groupoid with unit space $G_E^{(0)}=E^\infty$, and this groupoid is locally compact, Hausdorff and \'etale in a topology which has a neighbourhood basis consisting of the sets
\[
Z(\alpha,\beta)=\{(\alpha z,|\beta|-|\alpha|,\beta z):z\in E^\infty,\ r(z)=s(\alpha)\}
\]
parametrised by pairs of finite paths $\alpha,\beta\in E^*$ with $s(\alpha)=s(\beta)$ \cite[Proposition~2.6]{KPRR}. We know from \cite[Proposition~8.1]{HaH} that $G_E$ is principal if and only if $E$ has no cycles (in which case we say $E$ is \emph{acyclic}). 

We write $x\sim y$ to mean $(x,k,y)\in G_E$ for some $k$; this equivalence relation on $E^\infty$ is called \emph{tail equivalence with lag}.

\begin{prop}\label{Cartangraph}
Suppose that $E$ is an acyclic row-finite directed graph with no sources. Then $G_E$ is Cartan if and only if the quotient map $q:E^\infty\to E^\infty/\!\!\sim$ is a local homeomorphism.
\end{prop}

\begin{proof}
Suppose that $G_E$ is Cartan. The quotient space $E^\infty/\!\!\sim$ is the orbit space of $G_E$. Thus Proposition~\ref{prop:cartan_iff_pi_homeo} implies that $r\times s$ is an isomorphism of topological groupoids of $G_E$ onto $R(q)$. Since $G_E$ is \'etale, so is $R(q)$. Thus Lemma~\ref{lem:qtnt map} implies that $q$ is a local homeomorphism.

Conversely, suppose that $q$ is a local homeomorphism. We claim that $r\times s:G_E\to R(q)$ is an isomorphism of topological groupoids. Then, since we know from Proposition~\ref{prop:r(psi)_props} that $R(q)$ is Cartan, we can deduce that $G_E$ is Cartan too.

To prove the claim, we observe first that $r \times s$ is an isomorphism of algebraic groupoids because $G_E$ is principal, and is continuous because $r$ and $s$ are. So it suffices to take $\alpha,\beta\in E^*$ with $s(\alpha)=s(\beta)$, and prove that $r\times s(Z(\alpha,\beta))$ is open. A typical element of $r\times s(Z(\alpha,\beta))$ has the form $(\alpha z,\beta z)$ for some $z\in E^\infty$ with $r(z)=s(\alpha)$. Since $q$ is a local homeomorphism there is an initial segment $\mu$ of $z$ such that $q|_{Z(\mu)}$ is one-to-one. We will prove that $(Z(\alpha\mu)\times Z(\beta\mu))\cap R(q)$ is contained in $r\times s(Z(\alpha,\beta))$.

Let $(x,y)\in(Z(\alpha\mu)\times Z(\beta\mu))\cap R(q)$. Then there are $x', y'\in E^\infty$ such that $x=\alpha\mu x'$ and $y=\beta\mu y'$, and $q(x)=q(y)$ implies $q(\mu x')=q(\mu y')$. Both $\mu x'$ and $\mu y'$ are in $Z(\mu)$, so injectivity of $q|_{Z(\mu)}$ implies that $\mu x'=\mu y'$, and $x'=y'$. But now
\[
(x,y)=(\alpha\mu x',\beta\mu x')=r\times s(\alpha\mu x',|\beta|-|\alpha|,\beta\mu x')
\]
belongs to $r\times s(Z(\alpha,\beta))$. Thus $r\times s(Z(\alpha,\beta))$ contains a neighbourhood of $(\alpha z,\beta z)$, and $r\times s(Z(\alpha,\beta))$ is open. Now  we have proved our claim, and the result follows.
\end{proof} 

The groupoid $G_E$ was originally invented  as a groupoid whose $C^*$-algebra is the universal algebra $C^*(E)$ generated by a Cuntz-Krieger $E$-family \cite[Theorem~4.2]{KPRR}. In view of Corollary~\ref{cor:Fell_iff_pi_homeo}, Proposition~\ref{Cartangraph} implies that $C^*(E)=C^*(G_E)$ is a Fell algebra if and only if $q:E^\infty\to E^\infty/\!\!\sim$ is a local homeomorphism. So it is natural to ask whether we can identify this property at the level of the graph. We can:

\begin{prop}\label{prop-graphcriteria}
Suppose that $E$ is an acyclic row-finite directed graph with no sources. Then the quotient map $q:E^\infty\to E^\infty/\!\!\sim$ is a local homeomorphism if and only if, for every $x\in E^\infty$, there exists $n$ such that
\begin{equation}\label{charFell}
s(x_n)E^*s(\mu)=\{\mu\}\text{ for every $\mu\in E^*$ with $r(\mu)=s(x_n)$.}
\end{equation}
\end{prop}

\begin{proof}
Suppose first that $q$ is a local homeomorphism, and $x\in E^\infty$. Then there is an initial segment $\mu=x_1x_2\cdots x_n$ of $x$ such that $q|_{Z(\mu)}$ is injective. We claim that this $n$ satisfies \eqref{charFell}. Suppose not. Then there exist $\alpha,\beta\in s(x_n)E^*$ such that $\alpha\not=\beta$ and $s(\alpha)=s(\beta)$. Neither can be an initial segment of the other, since this would give a cycle at $s(\alpha)$. So there exists $i\leq \min(|\alpha|,|\beta|)$ such that $\alpha_i\not=\beta_i$. Then for any $y\in E^\infty$ with $r(y)=s(\alpha)$, $\mu\alpha y$ and $\mu\beta y$ are distinct paths in $Z(\mu)$ with $q(\mu\alpha y)=q(\mu\beta y)$, which contradicts the injectivity of $q|_{Z(\mu)}$.

Conversely, suppose that $E$ has the property described, and let $x\in E^\infty$. Take $n$ satisfying \eqref{charFell}, and $\mu:=x_1x_2\cdots x_n$. We claim that $q|_{Z(\mu)}$ is injective. Suppose $y=\mu y', z=\mu z'\in Z(\mu)$ and $q(y)=q(z)$. Then $q(y')=q(z')$, and there exist paths $\gamma,\delta$ in $s(x_n)E^*$ such that $y'=\gamma y''$, $z'=\delta y''$, say. The existence of $y''$ forces $s(\gamma)=s(\delta)$, and \eqref{charFell} implies that $\gamma=\delta$, $y'=z'$ and $y=z$. Thus $q$ is locally injective. Since $G_E$ is \'etale, the quotient map $q$ is open, and hence $q$ is a local homeomorphism.
\end{proof}

\begin{ex}\label{HaHgraph} Consider the following graph $E$ from \cite[Example~8.2]{HaH}:
\[
\begin{alternativegraphic}[2]{merged_fun_with_groupoids_graphics}
\begin{tikzpicture}[>=stealth,baseline=(current bounding box.center)] 
\def\cellwidth{5.5};
\clip (-5em,-5.6em) rectangle (3*\cellwidth em + 4.5em,0.3em);

\foreach \x in {1,2,3,4} \foreach \y in {0} \node (x\x y\y) at (\cellwidth*\x em-\cellwidth em,-3*\y em) {$\scriptstyle v_{\x}$};
\foreach \x in {1,2,3,4} \foreach \y in {1} \node (x\x y\y) at (\cellwidth*\x em-\cellwidth em,-3*\y em) {};
\foreach \x in {1,2,3,4} \foreach \y in {2} \node (x\x y\y) at (\cellwidth*\x em-\cellwidth em,-1.5em-1.5*\y em) {};
\foreach \x in {1,2,3,4} \foreach \y in {1,2} \fill[black] (x\x y\y) circle (0.15em);

\foreach \x in {1,2,3,4} \draw [<-, bend left] (x\x y0) to node[anchor=west] {$\scriptstyle f_{\x}^{(2)}$} (x\x y1);
\foreach \x in {1,2,3,4} \draw [<-, bend right] (x\x y0) to node[anchor=east] {$\scriptstyle f_{\x}^{(1)}$} (x\x y1);

\foreach \x / \z in {1/2,2/3,3/4} \draw[black,<-] (x\x y0) to node[anchor=south] {} (x\z y0);

\foreach \x in {1,2,3,4} \draw [<-] (x\x y1) -- (x\x y2);
\foreach \x in {1,2,3,4} {
	\node (endtail\x) at (\cellwidth*\x em-\cellwidth em, -6em) {};
	\draw [dotted,thick] (x\x y2) -- (endtail\x);
}

\node(endtailone) at (3*\cellwidth em + 2.5em,0em) {};
\draw[dotted,thick] (x4y0) -- (endtailone);
\end{tikzpicture}
\end{alternativegraphic}
\]
Let $G_E$ be the path groupoid. Let $z$ be the infinite path with range $v_1$ that passes through each $v_n$, and, for $n\geq 1$, let $x_n$ be the infinite path with range $v_1$ that includes the edge $f_n^{(1)}$.   It is shown in \cite[Example~8.2]{HaH} that the sequence $\{x_n\}$ ``converges $2$-times in $E^\infty/G_E$ to $z$'', and it then follows from \cite[Lemma~5.1]{CaH2} that $G_E$ is not Cartan.  Applying the criteria of Proposition~\ref{prop-graphcriteria} seems to give an easier proof of this: let $z$ be as above.  For each $n$, 
$s(z_n)=v_{n+1}$ and $s(z_n)E^*s(f_{n+1}^{(1)})=\{f_{n+1}^{(1)}, f_{n+1}^{(2)}\}$. Thus $q:E^\infty\to E^\infty/\!\!\sim$ is not a local homeomorphism by Proposition~\ref{prop-graphcriteria} and hence $G_E$ is not Cartan by Proposition~\ref{Cartangraph}.  Proposition~\ref{prop:cartan_iff_pi_homeo} implies that $r\times s$ is not open.
\end{ex}



\section{Fell algebras with trivial Dixmier-Douady invariant}\label{trivialDD}

If $A$ is a Fell algebra, we write $\delta(A)$ for its Dixmier-Douady invariant, as defined in \cite[Section~7]{aHKS}.
If $A$ is a continuous-trace algebra, then $\delta(A)$ makes sense, and $A$ also has a Dixmier-Douady invariant $\deltaDD(A)$ as in \cite[\S5.3]{tfb}, for example; as pointed out in \cite[Remark~7.10]{aHKS}, it is not clear whether these invariants are the same. Recall that the properties of being a Fell algebra or having continuous trace are preserved under Morita equivalence by  \cite[Corollary~14]{aHRW} and \cite[Corollary~3.5]{zettl2}.

The following is implicitly assumed in \cite{aHKS}.

\begin{thm}\label{thm-true}
Let $A$ be a separable Fell algebra.  Then the Dixmier-Douady invariant $\delta(A)$ of $A$ is $0$ if and only if there is is a local 
homeomorphism $\psi$ of a second-countable, locally compact and Hausdorff space onto a topological space such that $A$ is Morita equivalent to $C^*_{\red}(R(\psi))$.
\end{thm}

\begin{proof} We start by recalling how $\delta(A)$ is defined in \cite{aHKS}. By \cite[Theorem~5.17]{aHKS}, $A$ is Morita equivalent to a $C^*$-algebra $C$ which has a diagonal $C^*$-subalgebra $D$; let $h:\hat C\to \hat A$ be an associated Rieffel homeomorphism. By \cite[Theorem~3.1]{K}, there is a twist $\Gamma\to R$ over an \'etale and principal  groupoid $R$ such that $C$ is isomorphic to Kumjian's $C^*$-algebra $C^*(\Gamma;R)^{\Kum}$ of the twist. Since $C$ is a Fell algebra we may by \cite[Proposition~6.3]{aHKS} assume that $R=R(\psi)$, where $\psi:\hat D\to\hat C$ is the spectral map, which is a local homeomorphism by \cite[Theorem~5.14]{aHKS}. 

By \cite[Remark~2.9]{Kumjian2}, there is an extension $\underline{\Gamma}\to R(\psi)$ where $\underline{\Gamma}$ is the groupoid consisting of germs of continuous local sections of the surjection $\Gamma\to R(\psi)$. Such extensions are called sheaf twists, and the group of their isomorphism classes is denoted by $T_{R(\psi)}(\mathcal{S})$. Let $H^2(R(\psi),\mathcal{S})$ be the second equivariant sheaf cohomology group.  The long exact sequence of \cite[Theorem~3.7]{Kumjian2} yields a boundary map $\partial^1$ from $T_{R(\psi)}(\mathcal{S})$ to $H^2(R(\psi),\mathcal{S})$.
Finally, set   \[\delta(A)=(\pi^*_{h\circ\psi})^{-1}(\partial^{1}([\underline\Gamma]))\in H^2(\hat A, \mathcal{S})\]
where $\pi_{h\circ\psi}: R(h\circ\psi)\to\hat A$ is given by $(y,z)\mapsto h\circ\psi(y)$ \cite[Definition~7.9]{aHKS}. Quite a bit of the work in \cite[\S7]{aHKS} is to show that $\delta(A)$ is well-defined. 

Now  suppose that $\delta(A)=0$. Let $\Gamma$ be a twist associated to $A$. Then $\partial^{1}([\underline\Gamma])=0$. Let $\Lambda:=\T\times R(\psi)$ so that $\Lambda\to R(\psi)$ is the trivial twist.  Then the associated sheaf twist $\underline{\Lambda}$ is  also trivial, whence
$\partial^1([\underline{\Lambda}])=0$.  Now $0=\partial^{1}([\underline\Gamma])$ is sent to $0=\partial^{1}([\underline\Lambda])$ under a certain natural isomorphism (see \cite[Corollary~7.6]{aHKS}), and  \cite[Lemma~7.12]{aHKS} implies that  $\Gamma\to R(\psi)$ and $\Lambda\to R(\psi)$ are equivalent twists. Thus  $C^*(\Gamma;R(\psi))^{\Kum}$ and $C^*(\Lambda;R(\psi))^{\Kum}$ are Morita equivalent by \cite[Lemma~6.5]{aHKS}.  But now $A$ and $C^*(\Lambda;R(\psi))^{\Kum}$ are Morita equivalent. Since $\Lambda\to R(\psi)$ is trivial,  $C^*(\Lambda;R(\psi))^{\Kum}$ is isomorphic to $C^*_{\red}(R(\psi))$ by \cite[Lemma~A.1]{aHKS}. Thus $A$ and $C^*_{\red}(R(\psi))$ are Morita equivalent.

Conversely,  suppose that $A$ is Morita equivalent to $C^*_{\red}(R(\psi))$ for some $\psi$.  By Lemma~A.1 of \cite{aHKS}, $C^*_{\red}(R(\psi))$ is isomorphic to the $C^*$-algebra  of the trivial twist $\Lambda:=\T\times R(\psi)\to R(\psi)$.  The associated sheaf
twist $\underline{\Lambda}$ is  also trivial,  so
$\partial^1([\underline{\Lambda}])=0$, and $\delta(C^*_{\red}(R(\psi))) = 0$.  
Since $A$ and $C^*_{\red}(R(\psi))$ are Morita equivalent, by Theorem~7.13 of \cite{aHKS} there is a homeomorphism $k:\hat A\to C^*_{\red}(R(\psi))^\wedge$ such that the induced isomorphism $k^*:H^2(C^*_{\red}(R(\psi)^\wedge,\mathcal{S})\to H^2(\hat A,\mathcal{S})$ carries $0=\delta(C^*_{\red}(R(\psi)))$ to $\delta (A)$. Thus $\delta (A)=0$.
\end{proof}

\begin{prop}\label{prop delta is 0}
Let $A$ be a separable Fell algebra.  
\begin{enumerate}
\item\label{prop delta is 0-1} If $\delta(A)=0$ then $\deltaDD(I)=0$ for every ideal $I$ of $A$ with continuous trace.
\item\label{prop delta is 0-2} If $A$ has continuous trace and $\deltaDD(A)=0$, then $\delta(A)=0$.
\end{enumerate}
\end{prop}

\begin{proof}
 Let $B_1$ and $B_2$ be $C^*$-algebras with continuous trace and  paracompact spectrum $X$. 
Propositions~5.32 and 5.33 of \cite{tfb} together say that $\deltaDD(B_1)=\deltaDD(B_2)$ if and only if $B_1$ and $B_2$ are Morita equivalent.  
Below we consider an ideal $I$ in a separable Fell algebra $A$ such that $I$ has continuous trace.  Then  $\hat I$ is second-countable, locally compact and Hausdorff, and hence is $\sigma$-compact. By \cite[Proposition~1.7.11]{P}, for example, $\hat I$ is paracompact.  Thus Propositions~5.32 and 5.33 of \cite{tfb} apply to continuous-trace  $C^*$-algebras with spectrum homeomorphic to $\hat I$.

 \eqref{prop delta is 0-1}
Suppose  $\delta(A)=0$,  and let $I$ be an ideal of $A$ with continuous trace.   
By Theorem~\ref{thm-true}, $A$ is Morita equivalent to $C^*_{\red}(R(\psi))$ for  some local homeomorphism $\psi:Y\to X$. Then $C^*_{\red}(R(\psi))$  is also a Fell algebra, and $I$ is Morita equivalent to an ideal $J$ of $C^*_{\red}(R(\psi))$ with continuous trace.  By Corollary~\ref{thm standard groupoid},  $C^*_{\red}(R(\psi))=C^*(R(\psi))$.  Since $R(\psi)$ is principal and $C^*(R(\psi))$ is liminary, by \cite[Proposition~6.1]{C} there exists an open invariant subset $U$ of the unit space $Y$ of $R(\psi)$ such that $J$ is isomorphic to the $C^*$-algebra of $R(\psi)|_U:=\{\gamma\in R(\psi):r(\gamma),s(\gamma)\in U\}$.  Since $R(\psi)|_U$ is principal and $C^*(R(\psi)|_U)$ has continuous trace, $R(\psi)|_U$ is a proper groupoid by \cite[Theorem~2.3]{MW}, and $\deltaDD(C^*(R(\psi)|_U))=0$ by \cite[Proposition~2.2]{MW}.  Since $C^*(R(\psi)|_U)$ and $I$  are Morita equivalent, $\deltaDD(I)=0$  by \cite[Proposition~5.32]{tfb}.

\eqref{prop delta is 0-2} Suppose $\deltaDD(A)=0$. Then $\deltaDD(A)=\delta_{DD}(C_0(\hat A))$, and $A$ and $C_0(\hat A)$ are Morita equivalent by \cite[Proposition~5.33]{tfb}. But $C_0(\hat A)$ is isomorphic to $C^*(R(\psi))=C^*_{\red}(R(\psi))$ where $\psi:\hat A\to \hat A$ is the identity.  Thus $\delta(A)=0$ by Theorem~\ref{thm-true}.
\end{proof}

The following corollary is immediate from Proposition~\ref{prop delta is 0}.

\begin{cor}\label{cor:0iff0} Suppose $A$ is a separable $C^*$-algebra with continuous trace.  Then $\delta(A)=0$ if and only if $\deltaDD(A)=0$.
\end{cor}

\section{Examples}\label{examples}

A standard example of a Fell algebra which does not have continuous trace is the algebra
\[
A_3=\{f\in C([0,1],M_2(\C)): f(1)\text{ is diagonal}\}
\]
discussed in \cite[Example~A.25]{tfb}. Here we describe two variations on this construction. It seems clear to us that our constructions could be made much more general, for example by doubling up along topologically nontrivial subspaces rather than at a single point. 

\subsection{A Fell algebra with trivial Dixmier-Douady invariant}\label{sec-example}
We write $\{ \xi_i\}_{i\in\N}$ for the usual orthonormal basis in $\ell^2(\N)$ and $\Theta_{ij}$ for the rank-one operator $\Theta_{ij}:h\mapsto (h\,|\,  \xi_j) \xi_i$ on $\ell^2(\N)$. We write $\calK=\calK(\ell^2(\N))$, and  
\begin{equation}\label{defA}
A:=\{a\in C([0,1],\calK): a(1)\text{\ is diagonal in the sense that $(a(1) \xi_i\,|\,  \xi_j)=0$ for $i\not=j$}\}.
\end{equation}
We write $e_{ij}$ for the constant function $t\mapsto \Theta_{ij}$ in $C([0,1],\calK)$. 

We let $Y$ be the disjoint union $\bigsqcup_{i\in \N}[0,1]=\bigcup_{i\in \N}[0,1]\times\{i\}$. Then $Y$ is a locally compact Hausdorff space with the topology in which each $[0,1]\times\{i\}$ is open, closed and homeomorphic to $[0,1]$. There is an equivalence relation $\sim$ on $Y$ such that $(s,i)\sim (s,j)$ for all $i,j\in \N$ and $s\in [0,1)$, and the $(1,i)$ are equivalent only to themselves. 

We claim that the quotient map $\psi:Y\to X:=Y/\!\!\sim$ is open. 
To see this, it suffices to take an open set $U=W\times\{i\}$ contained in one level $[0,1]\times\{i\}$,  and see that $\psi(U)$ is open. By definition of the quotient topology, we have to show that $\psi^{-1}(\psi(U))$ is open. If $(1,i)$ is not in $U$, then  
\[
\psi^{-1}(\psi(U))=\{(s,j):(s,i)\in U,\;j\in \N\}=\bigcup_{j\in \N}W\times\{j\},
\] 
which is open. If $(1,i)\in U$, then
\[
\psi^{-1}(\psi(U))=(W\times\{i\})\cup\Big(\bigcup_{j\not=i}((W\setminus\{1\})\times\{j\})\Big),
\]
which is open. Thus $\psi$ is open, as claimed. Since $[0,1]\times\{i\}$ is open and $\psi|_{[0,1]\times\{i\}}$ is injective, $\psi$ is a surjective local homeomorphism.

We now consider $R(\psi)$, and write $V_{ij}$ for the subset $\big(([0,1]\times\{i\})\times([0,1]\times\{j\})\big)\cap R(\psi)$. Then for $i\not=j$, the map $\psi_{ij}:((s,i),(s,j))\mapsto s$ is a homeomorphism of $V_{ij}$ onto $[0,1)$; for $i=j$, the similarly defined $\psi_{ii}$ is a homeomorphism of $V_{ii}$ onto $[0,1]$. Thus for each $f\in C_c(R(\psi))$, the compact set $\supp f$ meets only finitely many $V_{ij}$. Define $f_{ij}(s)=f((s,i),(s,j))$. Then $f_{ij}\in C_c([0,1))$ (for $i\not=j$) and $f_{ii}\in C([0,1])$, and $f$ can be recovered as a finite  sum $\sum_{\{(i,j):\supp f\cap V_{ij}\not=\emptyset\}}(f_{ij}\circ\psi_{ij})\chi_{V_{ij}}$. By viewing functions in $C_c([0,1))$ as functions on $[0,1]$ which vanish at $1$, we can define $\rho:C_c(R(\psi))\to A$ by 
\begin{equation*}\label{defrho}
\rho(f)=\sum_{\{(i,j):\supp f\cap V_{ij}\not=\emptyset\}}f_{ij}e_{ij}.
\end{equation*}

\begin{prop}\label{idA}
The function $\rho:C_c(R(\psi))\to A$ extends to an isomorphism of $C^*(R(\psi))$ onto $A$.
\end{prop}

First we check that $\rho$ is a homomorphism on the convolution algebra $C_c(R(\psi))$. Let $f,g\in C_c(R(\psi))$, and observe that for each $s$,
\begin{equation}\label{compprod}
(f*g)((s,i),(s,j))=\sum_{k} f((s,i),(s,k))g((s,k),(s,j))=\sum_k f_{ik}(s)g_{kj}(s)
\end{equation}
has only finitely many nonzero terms (and just one if $s=1$ and $i=j$). Since the $e_{ij}$ are matrix units in $C([0,1],\calK)$, and the operations in $A$ are those of $C([0,1],\calK)$, Equation~\ref{compprod} implies that $\rho$ is multiplicative. Thus $\rho$ is a $*$-homomorphism. 

To see that $\rho$ extends to $C^*(R(\psi))$, we show that $\rho$ is isometric for the reduced norm on $C_c(R(\psi))$, which by Corollary~\ref{thm standard groupoid} is the same as the enveloping norm. The norm in $C([0,1],\calK)$ satisfies $\|a\|=\sup\{\|a(t)\|:t\in [0,1)\}$, and since $[0,1)$ is also dense in $X$, the reduced norm on $C_c(R(\psi))$ satisfies
\[
\|f\|=\sup\{\|\Ind_{(s,i)}(f)\|:s\in [0,1)\}.
\]
Thus the following lemma implies that $\rho$ is isometric.

\begin{lemma}\label{lemequiv}
For $s\in [0,1)$, we define $\epsilon_s:A\to B(\ell^2(\N))$ by $\epsilon_s(f)=f(s)$. Then the representation $\epsilon_s\circ\rho$ of $C_c(R(\psi))$ is unitarily equivalent to $\Ind_{(s,i)}$. 
\end{lemma}

\begin{proof}
For each $i\in \N$ we have $s^{-1}((s,i))=\{((s,j),(s,i)):j\in \N\}$, so there is a unitary isomorphism $U_s$ of $\ell^2(s^{-1}((s,i)))$ onto $\ell^2(\N)$ which carries the point mass $e_{((s,j),(s,i))}$ into the basis vector $\xi_j$. We will prove that $U_s$ intertwines $\epsilon_s\circ\rho$ and $\Ind_{(s,i)}$. Let $f\in C_c(R(\psi))$. On one hand, we have
\[
\epsilon_s\circ\rho(f)(U_se_{((s,j),(s,i))})=\sum_{k,l} f_{kl}(s)\Theta_{kl}\xi_j=\sum_{k} f_{kj}(s)\xi_k.
\]
On the other hand, the induced representation satisfies
\begin{align*}
\big(\Ind_{(s,i)}(f)e_{((s,j),(s,i))}\big)((s,k),(s,i))&=\sum_l f((s,k),(s,l))e_{((s,j),(s,i))}((s,l),(s,i))\\
&=f((s,k),(s,j))=f_{kj}(s), 
\end{align*}
and hence
\[
U_s\big(\Ind_{(s,i)}(f)e_{((s,j),(s,i))}\big)=U_s\Big(\sum_k f_{kj}(s)e_{((s,k),(s,i))}\Big) =\sum_k f_{kj}(s)\xi_k.\qedhere
\]
\end{proof}

It remains for us to see that $\rho$ is surjective.
Since $\rho(C^*(R(\psi)))$ is a $C^*$-algebra, it suffices to show that $\rho(C_c(R(\psi)))$ is dense in $A$. Indeed, because the $q_N=\sum_{i=1}^Ne_{ii}$ form an approximate identity for $A$, it suffices to take $b\in q_NAq_N$ and show that we can approximate $b$ by some $\rho(f)$. Fix $\epsilon>0$. We can write $b=\sum_{i,j\leq N}b_{ij}e_{ij}$ with $b_{ij}\in C([0,1])$ and $b_{ij}(1)=0$ for $i\not=j$. Set $c:=b-\sum_{i=0}^N b_{ii}e_{ii}$, and observe that $c(1)=0$.  Choose $\delta\in(0,1)$ such that $\sup_{t\in[\delta,1)}\|c(t)\|<\epsilon$, and choose a continuous function $h:[0,1]\to [0,1]$ such that $h=1$ on $[0,\delta]$ and $h=0$ near $1$.  Then $\|c-hc\|_\infty<\epsilon$.   Set $d:=\sum_{i=0}^N b_{ii}e_{ii}+hc$, and then $d$ has the form $\sum_{i=0}^N d_{ij}e_{ij}$,
where $d_{ij}\in C_c([0,1))$ if $i\neq j$ and $d_{ii}\in C([0,1])$. Set $f=\sum_{i=0}^N d_{ij}\chi_{V_{ij}}$. Then  $f\in C_c(R(\psi))$, $\rho(f)=d$,  and
\[\|b-\rho(f)\|_\infty=\Big\|b-\sum_{i=1}^n b_{ii}e_{ii}-hc\Big\|_\infty=\|c-hc\|_\infty<\epsilon.\] 

This completes the proof of Proposition~\ref{idA}. 

\begin{cor}
The $C^*$-algebra $A$ in \eqref{defA} is a Fell algebra which does not have continuous trace, and the Dixmier-Douady invariant of $A$ is $0$.
\end{cor}

\begin{proof}  By Proposition~\ref{idA}, $A$ is isomorphic to $C^*(R(\psi))$ where $\psi$ is a surjective local homeomorphism. Now Corollary~\ref{thm standard groupoid} implies that $A$ is a Fell algebra with spectrum $X$, and Theorem~\ref{thm-true} implies that its Dixmier-Douady invariant vanishes. Because $X$ is not Hausdorff, $A$ does not have continuous trace.
\end{proof}

\begin{rmk}
Paracompactness is usually defined only for Hausdorff spaces, and the example of this section confirms that things can go badly wrong for the sorts of non-Hausdorff spaces of interest to us. For the spectrum $X$ of our algebra $A$, the sets $\psi([0,1]\times\{i\})$ form an open cover of $X$, but every neighbourhood of the point $\psi(1,1)$, for example, meets every neighbourhood of every other $\psi(1,i)$, so there cannot be a locally finite refinement.
\end{rmk}

\subsection{A Fell algebra with non-trivial Dixmier-Douady invariant}

We describe a Fell algebra $A$ which does not have continuous trace, and  has Dixmier-Douady class $\delta(A)\not=0$. We do this by combining the construction of \cite[\S1]{RT} (see also \cite[Example~5.23]{tfb}) with that of the algebra $A_3$ at the start of \S\ref{examples}. We adopt the notation of \cite[Chapter~5]{tfb}.

We start with a compact Hausdorff space $S$,  a finite open cover  $\mathcal{U}=\{U_1,\ldots,U_n\}$ of $S$, and  an alternating cocycle $\lambda_{ijk}:U_{ijk}\to \T$ whose class $[\lambda_{ijk}]$ in $H^2(S,\mathcal{S})$ is nonzero. By the argument of \cite[Lemma~3.4]{RW}, for example, we may multiply $\lambda$ by a coboundary and assume  that $\lambda_{ijk}\equiv 1$ whenever two of $i,j,k$ coincide.

The algebra $A(\mathcal{U},\lambda_{ijk})$ in \cite[Example~5.23]{tfb} has underlying vector space
\[
A(\mathcal{U},\lambda_{ijk})=\{(f_{ij})\in M_n(C(S)): f_{ij}=0\text{ on }S\setminus U_{ij}\}.
\]
The product in $A(\mathcal{U},\lambda_{ijk})$ is defined by $(f_{ij})(g_{kl})=(h_{il})$, where
\begin{equation}\label{RTprod}
h_{il}(s)=\begin{cases}
\sum_{\{j:s\in U_{ijl}\}}\overline{\lambda_{ijl}(s)}f_{ij}(s)g_{jl}(s)&\text{if  $s\in U_{ijl}$ for some $j$}\\
0&\text{otherwise,}
\end{cases}
\end{equation}
and the involution given by $(f_{ij})^*=(\overline{f_{ji}})$. For $s\in U$, we take $I_s:=\{i:s\in U_i\}$, and define $\pi_{i,s}:A(\mathcal{U},\lambda_{ijk})\to M_{I_s}$ by
\begin{equation}\label{RTrep}
\pi_{i,s}\big(\,(f_{jk})\,\big)=\big(\overline{\lambda_{ijk}(s)}f_{jk}(s)\big).
\end{equation}
It is shown in \cite[Example~5.23]{tfb} that $A(\mathcal{U},\lambda_{ijk})$ is a $C^*$-algebra with
\[
\big\|(f_{jk})\big\|=\sup_{i,s}\big\|\pi_{i,s}\big(\,(f_{jk})\,\big)\big\|.
\]
In fact, and we shall need this later, $A(\mathcal{U},\lambda_{ijk})$ is a continuous-trace algebra with spectrum $S$ and Dixmier-Douady class $\deltaDD(A(\mathcal{U},\lambda_{ijk}))=[\lambda_{ijk}]$ (see \cite[Proposition~5.40]{tfb}, which simplifies in our case because $S$ is compact and the cover is finite).

For our new construction, we fix a point $*$ in $U_1$, and suppose that $*$ is not in any other $U_i$ (we can ensure this is the case by replacing $U_i$ with $U_i\setminus \{*\}$). We add a copy $U_0$ of $U_1$ to our cover, and set
\[
Y:=\bigsqcup_{i=0}^n U_i=\bigcup_{i=0}^n\big\{(s,i):0\leq i\leq n\text{ and }s\in U_i\big\}.
\]
We define a relation $\sim$ on $Y$ by $(s,i)\sim(s,j)$ if  $s\not=*$, $(*,1)\sim(*,1)$, and $(*,0)\sim(*,0)$. This is an equivalence relation, and we define $X$ to be the quotient space and $\psi:Y\to X$ to be the quotient map. Thus $X$ consists of a copy of $S\setminus\{*\}$ with the subspace topology, and two closed points $\psi(*,0)$, $\psi(*,1)$ whose open neighbourhoods are the images under $\psi$ of open sets $U\times\{0\}$ and $U\times\{1\}$, respectively. 

\begin{lemma}
The function $\psi:Y\to X$ is a surjective local homeomorphism.
\end{lemma}

\begin{proof}
Quotient maps are always continuous and surjective, and $\psi$ is injective on each $U_i\times\{i\}$. So it suffices to see that $\psi$ is open, and for this, it suffices to see that for each open set $W$ in $U_i$, $\psi(W\times\{i\})$ is open in $X$. By definition of the quotient topology, we need to show that $\psi^{-1}(\psi(W\times\{i\}))$ is open in $Y$. If $*$ is not in $W$, then
\[
\psi^{-1}(\psi(W\times\{i\}))=\bigcup_{\{j:U_j\cap W\not=\emptyset\}}(W\cap U_j)\times\{j\}.
\]
If $*\in W$ and $i=0$, then 
\[
\psi^{-1}(\psi(W\times\{0\}))=\{W\times\{0\}\}\cup\Big(\bigcup_{\{j:U_j\cap W\not=\emptyset\}}((W\setminus\{*\})\cap U_j)\times\{j\}\Big)
\]
is open, and similarly for $*\in W$ and $i=1$. So $\psi^{-1}(\psi(W\times\{i\}))$ is always open, as required.
\end{proof}

 We extend $\lambda$ to an alternating cocyle on the cover $\{U_0,U_1,\cdots,U_n\}$ by setting $\lambda_{0jk}=\lambda_{1jk}$. Then the formula
\[
\sigma\big(((s,i)(s,j)),((s,j)(s,k)) \big)=\overline{\lambda_{ijk}(s)}
\]
defines a continuous $2$-cocycle $\sigma$ on $R(\psi)$. Since $\psi$ is a local homeomorphism it follows from Proposition~\ref{prop:r(psi)_props} that $R(\psi):=\{(y,z)\in Y\times Y:\psi(y)=\psi(z)\}$ is a locally compact, Hausdorff and \'etale groupoid, and that 
$C^*(R(\psi),\sigma)=C^*_{\red}(R(\psi),\sigma)$ is a Fell algebra with spectrum homeomorphic to $X$.

The $*$-algebra structure on $C_c(R(\psi),\sigma)$ is given by
\begin{align}
f^*((s,i),(s,j))&=\overline{f((s,j),(s,i))\lambda_{iji}(s)}=\overline{f((s,j),(s,i))}\notag\\
(f*g)((s,i),(s,j))&=\sum_{\{k:s\in U_k\}}f((s,i),(s,k))g((s,k),(s,j))\overline{\lambda_{ikj}(s)},
\label{gpoidprod}
\end{align}
and if $(s,i)$ is a unit in $R(\psi)$ then the induced representation $\Ind^\sigma_{(s,i)}$ acts in $\ell^2(s^{-1}((s,i)))=\ell^2(\{(s,j):s\in U_{ij}\})$ according to the formula 
\begin{equation}\label{defind}
(\Ind^\sigma_{(s,i)}(f)\xi)(s,j)=\sum_{\{k:s\in U_k\}}f((s,j),(s,k))\xi(s,k)\overline{\lambda_{jki}(s)}.
\end{equation}

\begin{lemma}\label{pi0}
There is a homomorphism $\pi_0:C^*(R(\psi),\sigma)\to \C$ such that $\pi_0(f)=f((*,0),(*,0))$ for $f\in C_c(R(\psi),\sigma)$.
\end{lemma}

\begin{proof}
The inverse image $s^{-1}((*,0))$ consists of the single point $((*,0),(*,0))$, so the Hilbert space $\ell^2(s^{-1}(*,0))$ is one dimensional, and $\Ind^\sigma_{(*,0)}(f)$  is multiplication by the complex number
\begin{align*}
f((*,0),(*,0))\overline{\lambda_{000}(1)}=f((*,0),(*,0)).
\end{align*}
In other words, the representation $\Ind^\sigma_{(*,0)}$ of $C_c(R(\psi),\sigma)$ has the property we require of $\pi_0$. Since the reduced norm is $f\mapsto \sup_{y\in Y}\{\|\Ind^\sigma_y(f)\|\}$, $\Ind^\sigma_{(*,0)}$ is bounded for the reduced norm, and extends to a representation on $C^*_{\red}(R(\psi),\sigma)=C^*(R(\psi),\sigma)$.
\end{proof}

\begin{lemma}
Define $V_0:=U_0\setminus\{*\}$, $V_i:=U_i$ for $i\geq 1$ and $\mathcal{V}:=\{V_0,V_1,\cdots, V_n\}$. Then the ideal $\ker \pi_0$ is isomorphic to the $C^*$-algebra $A(\mathcal{V},\lambda_{ijk})$ of \cite[Example 5.23]{tfb}.
\end{lemma}

\begin{proof}
The maps $\phi_{ij}:(V_i\times\{i\})\times_\psi(V_j\times \{j\})\to S$ defined by $\phi_{ij}:((s,i),(s,j))\mapsto s$ are homeomorphisms of $(V_i\times\{i\})\times_\psi(V_j\times \{j\})$ onto $V_{ij}$. Thus for $f\in C_c(R(\psi))$ such that $\pi_0(f)=0$, we can define $f_{ij}:V_{ij}\to \C$ by $f_{ij}=f\circ \phi_{ij}^{-1}$. For $\{i,j\}\not=\{0,1\}$, the function $f_{ij}$ has compact support, and extends uniquely to a continuous function $f_{ij}$ on $S$ with support in $V_{ij}$; because $f((*,0),(*,0))=\pi_0(f)=0$, the function $f_{01}$ vanishes on the boundary of $V_{01}=V_0$, and extends to a continuous function $f_{01}$ on $S$ which vanishes off $V_{01}$. 

At this point, we have constructed a map $\phi:f\mapsto (f_{ij})$ of $I_0:=C_c(R(\psi))\cap \ker\pi_0$ into the underlying set of $A(\mathcal{V},\lambda_{ijk})$. Since $f_{ij}(s)=f((s,i),(s,j))$, a comparison of \eqref{gpoidprod} with \eqref{RTprod} shows that $\phi$ is a homomorphism. It is also $*$-preserving. If $s\in V_i$, then a comparison of \eqref{defind} with \eqref{RTrep}, and an argument similar to the proof of Lemma~\ref{lemequiv}, show that $\pi_{i,s}\circ\phi$ is unitarily equivalent to the representation $\Ind^\sigma_{(s,i)}$. Thus $\phi$ is isometric for the reduced norm on $I_0$, and since the range is dense in $A(\mathcal{V},\lambda_{ijk})$, $\phi$ extends to an isomorphism of the closure $\ker\pi_0$ onto $A(\mathcal{V},\lambda_{ijk})$.
\end{proof}

\begin{rmk}
If we delete the point $(*,0)$ from $X$, we recover the original space $S$, the groupoid $R(\psi)$ is the one associated to the cover $\mathcal{V}$ of $S$ in Remark 3 on page~399 of \cite{RT}, and the isomorphism $\Phi$ of $C_{\red}^*(R(\psi),\sigma)$ with $A(\mathcal{V},\lambda_{ijk})$ is discussed in that remark. 
\end{rmk} 

Theorem ~1 of \cite{RT} (or Proposition~5.40 of \cite{tfb}) implies that $A(\mathcal{V},\lambda_{ijk})$ is a continuous-trace algebra with Dixmier-Douady class $\deltaDD(A(\mathcal{V},\lambda_{ijk}))=[\lambda_{ijk}]\not=0$. This implies that the ideal $\ker\pi_0$ in $C^*(R(\psi),\sigma)$ has $\deltaDD(\ker\pi_0)\not=0$. Now Proposition~\ref{prop delta is 0} implies that $\delta(C^*(R(\psi),\sigma))\not=0$.

\subsection{Epilogue}

We started this project looking for a cocycle-based
version of the Dixmier-Douady invariant of \cite{aHKS}, which would enable us to resolve the issue about compatibility of $\delta(A)$ and
$\deltaDD(A)$ in \cite[Remark~7.10]{aHKS}, and to 
construct concrete families of Fell algebras as in \cite{RT}. Since the spectrum $X$ of a Fell algebra is
locally locally-compact and locally Hausdorff, it always has covers by
open Hausdorff subsets such that the overlaps $U_{ij}$ where cocycles live
lie inside large Hausdorff subsets of $X$. So it seemed reasonable that
cocycle-based arguments might work.

As we progressed, we realised how crucially the steps by which one refines
covers, as in the proof of \cite[Proposition~5.24]{tfb}, for example,
depend on the existence of locally finite refinements. In the example of
\S\ref{sec-example}, this local finiteness fails spectacularly. So even
though we know that that the algebra in \S\ref{sec-example} is a Fell
algebra, and even though we know it must have vanishing Dixmier-Douady
invariant, it is hard to see how a cocycle-based theory could accommodate it.

The second part of our project has worked to some extent, in that we can
see how to build lots of Fell algebras from ordinary \v{C}ech cocycles.
However, we can also see that the possibilities are almost limitless, and
at this stage there seems little hope of finding a computable invariant.

\appendix\label{app}
\section{Twisted groupoid $C^*$-algebras}\label{app-twists2}

There are several different ways of twisting the construction of a groupoid $C^*$-algebra. They include: 
\begin{enumerate} 
\item\label{a} Renault's $C^*(G, \sigma)$ associated to a $2$-cocycle $\sigma:G^{(2)}\to \T$ on a groupoid $G$ from \cite[II.1]{Ren} (which we discuss in \S\ref{sec-back} and use in \S\ref{Lisa} and \S\ref{examples});
\item
Kumjian's $C^*(\Gamma;G)^{{\Kum}}$ associated to a twist $\Gamma$ over a principal, \'etale groupoid $G$ in \cite[\S2]{K} (which we use in \S\ref{trivialDD});
\item\label{c} Muhly and Williams' $C^*(E;G)^{{\MW}}$ associated to an extension  $E$ of a  groupoid $G$ by $\T$ in \cite[\S2]{MW} (which we use in the proof of Theorem~\ref{Cartan=Fell}); 
\item the reduced $C^*$-algebras $C^*_{\red}(G, \sigma)$ and  $C_{\red}^*(E;G)^{{\MW}}$ corresponding to  \eqref{a} and \eqref{c}, respectively.
\end{enumerate} 
Here we only consider second-countable, locally compact, Hausdorff and principal group\-oids $G$ with a left Haar system $\lambda=\{\lambda^u\}$. Let $\sigma:G^{(2)}\to\T$ be a continuous normalised $2$-cocycle on $G$. Following~\cite[page~73]{Ren} we denote by $G^\sigma$ the associated extension of $G$ by $\T$:  thus   $G^\sigma$ is the groupoid $\T\times G$  with the product topology, with  range and source  maps $r(z,\alpha)=(1,r(\alpha))$ and $s(z,\alpha)=(1,s(\alpha))$, multiplication $(w,\alpha)(z,\beta)=(wz\sigma(\alpha,\beta),\alpha\beta)$ and inverse $(z,\alpha)^{-1}=(z^{-1}\sigma(\alpha,\alpha^{-1})^{-1},\alpha^{-1})$.    Then $G^\sigma$ is a locally compact Hausdorff groupoid with left Haar system $\sigma=\{\sigma^u\}$, where $\sigma^u$ is the product of the normalised  Haar measure on $\T$ and $\lambda^u$. 

A twist $\Gamma$  over a principal \'etale  groupoid $G$ has an underlying principal $\T$-bundle over $G$, and in \cite[page~985]{K} Kumjian observes that the twists whose underlying bundle is trivial are in one-to-one correspondence with  continuous $2$-cocycles $\sigma$. 

Set
\[
C_c(G^\sigma;G):=\{f\in C_c(G^\sigma): f(z\cdot \gamma)=zf(\gamma)\text{ for $z\in \T$, $\gamma\in G^\sigma$}\}.
\]
It is easy to check  that $C_c(G^\sigma;G)$ is a $*$-subalgebra of the usual convolution algebra $C_c(G^\sigma)$. The $C^*$-algebra $C^*(G^\sigma;G)^{\MW}$ is by definition the completion of $C_c(G^\sigma;G)$ in the supremum norm $\|f\|=\sup\{\|L(f)\|\}$, where $L$ ranges over a collection of appropriately continuous $*$-representations of $C_c(G^\sigma;G)$. By \cite[Proposition~3.7]{BaH}, $C^*(G^\sigma;G)^{\MW}$ is a direct summand of $C^*(G^\sigma)$.

We know from \cite[Lemma~3.1]{BaH} that the map $\rho:C_c(G^\sigma; G)\to C_c(G,\overline{\sigma})$ defined by $\rho(f)(\alpha)=f(1,\alpha)$ is a $*$-isomorphism of $C_c(G^\sigma;G)$ onto Renault's twisted convolution algebra $C_c(G,\overline{\sigma})$ (the algebra $C_c(G^\sigma;G)$ is denoted $C_c(G^\sigma,-1)$ in \cite{BaH}), and that $\rho$ extends to an isomorphism of $C^*(G^\sigma;G)^{\MW}$ onto $C^*(G,\overline{\sigma})$. 

Let $u\in G^{(0)}$. For the proof of Theorem~\ref{Cartan=Fell} we need to know that the isomorphism $\rho$ sends the class of certain irreducible representations $L^u$ of $C^*(G^\sigma;G)^{\MW}$ defined in \cite[\S3]{MW} to the class of the representation $\Ind_u^{\overline\sigma}$ of $C^*(G,\overline{\sigma})$.  The Hilbert space of $L^u$ is  the completion $H_u$ of $H_u^0:=\{g\in C_c(G^\sigma; G):\supp g\subset \T\times s^{-1}(u)\}$ with respect to the inner product $(f\,|\,g)=\int_Gf(1,\alpha)\overline{g(1,\alpha)}\, d\lambda_u(\alpha)$. Let  $f\in C_c(G^\sigma;G)$ and $g\in H_u^0$. Then $L^u(f)g=f*g$, where the convolution takes place in $C_c(G^\sigma; G)$.  We compute using the formulas in \cite[Remark~2.3]{BaH} and the cocycle identity for the triple $(\beta,\beta^{-1},\alpha)$:
\begin{align}
(L^u(f)g)(1,\alpha)&=\int_G f(1,\beta)g\big((1,\beta)^{-1}(1,\alpha)\big)\, d\lambda^{r(\alpha)}(\beta)\notag\\
&=\int_G f(1,\beta)g\big(\overline{\sigma(\beta,\beta^{-1})}\sigma(\beta^{-1},\alpha),\beta^{-1}\alpha\big)\, d\lambda^{r(\alpha)}(\beta)\notag\\
&=\int_G f(1,\beta)g\big(\overline{\sigma(\beta,\beta^{-1}\alpha)}\sigma(\beta\beta^{-1},\alpha),\beta^{-1}\alpha\big)\, d\lambda^{r(\alpha)}(\beta)\notag\\
&=\int_G f(1,\beta)g(1,\beta^{-1}\alpha)\overline{\sigma(\beta,\beta^{-1}\alpha)}\, d\lambda^{r(\alpha)}(\beta)\label{eq-thesame}.
\end{align}
Let $U:C_c(G^\sigma;G)\to L^2(s^{-1}(u),\lambda_u)$ be the map defined by $(Ug)(\alpha)=g(1,\alpha)$. Then $U$ extends to a unitary $U$ from the Hilbert space $H_u$ of $L^u$ onto the Hilbert space $L^2(s^{-1}(u),\lambda_u)$ of $\Ind_u^{\overline\sigma}$. We have
\begin{align*}
(\Ind_u^{\overline\sigma}(\rho(f))U(g))(\alpha)&=\int_G\rho(f)(\beta)U(g)(\beta^{-1}\alpha)\overline{\sigma(\beta,\beta^{-1}\alpha)}\, d\lambda^{r(\alpha)}(\beta)\\
&=\int_G f(1,\beta)g(1, \beta^{-1}\alpha)\overline{\sigma(\beta,\beta^{-1}\alpha)}\, d\lambda^{r(\alpha)}(\beta),
\end{align*}
 which is $(U(L^u(f)g))(\alpha)$ using \eqref{eq-thesame}.  We have proved:
 
 \begin{lemma}\label{lem-rho} Let $G$ be a principal groupoid with Haar system and $\sigma:G^{(2)}\to\T$ a continuous normalised $2$-cocycle. 
Then for each unit $u$, $\Ind_u^{\overline\sigma}\circ\rho$ is unitarily equivalent to  $L^u$.
\end{lemma}
 
Now let $G$ be a principal and \'etale groupoid. It is shown in \cite[pages~977--8]{K} that there is a positive-definite $C_0(G^{(0)})$-valued inner product on $C_c(G^\sigma;G)$, and that the action of $C_c(G^\sigma;G)$ by left multiplication on itself extends to an action by adjointable operators on the Hilbert-module completion $H(G^\sigma)$. Then the $C^*$-algebra $C^*(G^\sigma;G)^{\Kum}$ is by definition the completion of $C_c(G^\sigma;G)$ in $\LL(H(G^\sigma))$. 

\begin{prop}\label{checktwists} Let $G$ be an \'etale and principal groupoid, and $\sigma:G^{(2)}\to\T$ a continuous normalised $2$-cocycle. 
The homomorphism $\rho$  of $C_c(G^\sigma; G)$ onto $C_c(G,\overline{\sigma})$ extends to an isomorphism of  $C^*(G^\sigma;G)^{\Kum}$  onto the reduced crossed product $C_{\red}^*(G,\overline{\sigma})$.
\end{prop}

In view of what we already know about $\rho$ from \cite{BaH}, it suffices to check that $\rho$ is isometric for the given norm on $C_c(G^\sigma)\subset \LL(H(G^\sigma))$ and the reduced norm on $C_c(G,\overline{\sigma})$. The general theory of Hilbert bimodules says that, if $\pi$ is a faithful representation of $C_0(G^{(0)})$, then the induced representation $H(G^\sigma)\dashind \pi$ is faithful on $\LL(H(G^\sigma))$. We can in particular take $\pi$ to be the atomic representation $\bigoplus_{u\in G^{(0)}}\epsilon_u$, and then $H(G^\sigma)\dashind \pi= \bigoplus_{u\in G^{(0)}}H(G^\sigma)\dashind \epsilon_u$. So Proposition~\ref{checktwists} follows from the following lemma.

\begin{lemma}
Let $G$ be an \'etale and principal groupoid, and $\sigma:G^{(2)}\to\T$ a continuous normalised $2$-cocycle.  For each $u\in G^{(0)}$, the representation $H(G^\sigma)\dashind \epsilon_u$ is unitarily equivalent to the representation $(\Ind_u^{\overline{\sigma}})\circ\rho$ of $C_c(G^\sigma;G)$.
\end{lemma}

\begin{proof}
Since $G$ is \'etale, the representation $\Ind_u^{\overline{\sigma}}$ of $C^*(G,\overline{\sigma})$ acts on $\ell^2(s^{-1}(u))$ by the formula
\begin{equation}\label{Indu}
(\Ind_u^{\overline{\sigma}}(h)\xi)(\alpha)=\sum_{r(\beta)=r(\alpha)} h(\beta)\xi(\beta^{-1}\alpha)\overline{\sigma(\beta,\beta^{-1}\alpha)}
\end{equation}
for $h\in C_c(G)$, $\xi\in\ell^2(s^{-1}(u))$, $\alpha\in G$. The representation $H(G^\sigma)\dashind \epsilon_u$ acts on (the completion of) $C_c(G^\sigma;G)\otimes_{C_0(G^{(0)})}\C$, which is $C_c(G^\sigma;G)$ with the inner product $(f\,|\,g)=(g^**f)(u)$. The Haar system on $G^\sigma$ is the product of the normalised Haar measure on $\T$ and the counting measure on $s^{-1}(u)$. Let  $f,g\in C_c(G^\sigma;G)$.  We have
\begin{align*}
(f\,|\,g)=(g^**f)(u)&=\int_{\T}\Big(\sum_{s(\alpha)=u}g^*(\overline{z},\alpha^{-1})f(z,\alpha)\Big)\,dz
=\int_{\T}\Big(\sum_{s(\alpha)=u}\overline{g(z,\alpha)}f(z,\alpha)\Big)\,dz\\
&=\int_{\T}\Big(\sum_{s(\alpha)=u}\overline{zg(1,\alpha)}zf(1,\alpha)\Big)\,dz=\sum_{s(\alpha)=u}f(1,\alpha)\overline{g(1,\alpha)},
\end{align*}
and it follows that the Hilbert space of $H(G^\sigma)\dashind \epsilon_u$ is the space $H_u$. The action of $H(G^\sigma)\dashind \epsilon_u(f)$ on $g$ is by multiplication, so $H(G^\sigma)\dashind \epsilon_u(f)=L^u(f)$. Now the result follows from Lemma~\ref{lem-rho}.
\end{proof}

\begin{cor}\label{Lisa2}
Suppose that $G$ is a second-countable, locally compact, Hausdorff, \'etale, principal and Cartan groupoid, and that $\sigma$ is a continuous normalised $2$-cocycle on $G$. Then
\begin{equation}\label{Lisa'send}
C^*(G^\sigma;G)^{\MW}\cong C^*(G,\overline{\sigma})=C_r^*(G,\overline{\sigma})\cong C^*(G^\sigma;G)^{\Kum}.
\end{equation}
\end{cor}

\begin{proof}
The first isomorphism is from \cite[Lemma~3.1]{BaH}, the equality in the middle is from Theorem~\ref{Cartan=Fell}, and the second isomorphism is from Proposition~\ref{checktwists}.
\end{proof}

\begin{rmk}
We can relax the hypothesis ``$G$ is Cartan'' in Corollary~\ref{Lisa2}. 
That hypothesis was used in the proof of Theorem~\ref{Cartan=Fell} to see that $u\mapsto [L^u]$ is a homeomorphism of $\go/G$ onto $(C^*(G^\sigma;G)^{\MW})^\wedge$; then Lemma~\ref{lem-rho} implies that every irreducible representation of $C^*(G,\sigma)$ is induced, and we have equality in the middle of \eqref{Lisa'send}. However, if we merely know that ``$\go/G$ is T$_0$'', then we can use \cite[Theorem~3.4]{CaH2} in place of \cite[Proposition~3.2]{CaH2} in the proof of Theorem~\ref{Cartan=Fell}, still deduce that $u\mapsto [L^u]$ is a homeomorphism, and follow the rest of the argument to get $C^*(G,\overline{\sigma})=C_r^*(G,\overline{\sigma})$.
\end{rmk}

\section{The ideal of continuous-trace elements}

If $A$ is a $C^*$-algebra, we write $\frak{m}(A)$ for the ideal spanned by the positive elements $a$ such that $\pi\mapsto\tr(\pi(a))$ is continuous on $\hat A$, as in \cite[\S4.5.2]{D}. The closure of $\frak{m}(A)$ is an ideal in $A$, which we call the \emph{ideal of continuous-trace elements}. 

\begin{prop}\label{openH}
Suppose that $A$ is a $C^*$-algebra. Then
\[
U:=\{\pi\in\hat A:\text{$\pi$ has a closed Hausdorff neighbourhood}\}
\]
is an open Hausdorff subset of $\hat A$.
\end{prop}

\begin{proof}
Let $x\in U$, and choose a closed Hausdorff neighbourhood $N$ of $x$. Then for each point $y$ in the interior $\inter N$, $N$ is a closed Hausdorff neighbourhood of $y$, and hence $y\in U$. Thus $U$ is open.

To see that $U$ is Hausdorff, let $x,y\in U$, and choose closed Hausdorff neighbourhoods $N_x$ of $x$ and $N_y$ of $y$. If $x\in \inter N_y$, then since $N_y$ is Hausdorff we can choose open sets $V_x$ and $V_y$ in $N_y$ such that $x\in V_x$, $y\in V_y$ and $V_x\cap V_y=\emptyset$;  then $W_x:=V_x\cap(\inter N_y)$ and $W_y=V_y\cap(\inter N_y)$ are open subsets of $\inter N_y$ with the same property, and they are open in $\hat A$. A similar argument works if $y\in N_x$. It remains to deal with the case where $x\notin \inter N_y$ and $y\notin \inter N_x$. Since $\hat A$ is locally locally-compact, $\inter N_y$ contains a compact neighbourhood $K$ of $y$. Since $N_y$ is Hausdorff, $K$ is closed in $N_y$, and since $N_y$ is closed in $\hat A$, $K$ is closed in $\hat A$. Now $(\inter N_x)\setminus K$ is open in $\hat A$, and contains $x$ because $x\notin \inter N_y$. Thus $(\inter N_x)\setminus K$ and $\inter K$ are disjoint open neighbourhoods of $x$ and $y$.
\end{proof}

\begin{cor}\label{corFellpts}
Suppose that $A$ is a $C^*$-algebra. Then
\[
V:=\{\pi\in\hat A:\pi\text{ is a Fell point and has a closed Hausdorff neighbourhood}\}
\]
is an open Hausdorff subset of $\hat A$.
\end{cor}

\begin{proof}
The set of Fell points is an open subset of $\hat A$, so $V$ is the intersection of two open sets. It is Hausdorff because it is a subset of the Hausdorff set $U$ of Proposition~\ref{openH}.
\end{proof}

\begin{cor}
Suppose that $A$ is a $C^*$-algebra. Then the set $V$ of Corollary~\ref{corFellpts} is the spectrum of the ideal of continuous-trace elements of $A$.
\end{cor}

\begin{proof}
We show that $\pi\in V$ if and only if $\pi$ satisfies conditions (i) and (ii) of \cite[Lemma~2.2]{aHW}, and then our corollary follows from that lemma. 

First suppose that $\pi\in V$. Then since $V$ is open it is the spectrum of an ideal $J$ of $A$. 
Since  $\pi$ is a Fell point of $A$ it is also a Fell point of $J$.  To see this, let $a\in A^+$ and $W$ an open neighbourhood of $\pi$ in $\hat A$ such that $\rho(a)$ is a rank-one projection for all $\rho\in W$.  Let $f\in C_c(\hat J)^+$ such that $f$ is identically one on a neighbourhood $W'$ of $\pi$ in $\hat J$.  Using the Dauns-Hofmann theorem,  $f\cdot a$ is in $J^+$, and $\rho(f\cdot a)=f(\rho)\rho(a)=\rho(a)$ is a rank-one projection for $\rho\in W\cap W'$.
Thus $J$ is a continuous-trace algebra by \cite[Proposition~4.5.4]{D}, and $\pi$ satisfies (i). 

To verify (ii), note that $\pi$ has a closed Hausdorff neighbourhood $N$, and a neighbourhood base of compact sets \cite[Corollary~3.3.8]{D}. Since $\pi\in \inter N$, the compact neighbourhoods $K$ with $K\subset\inter N$ also form a neighbourhood base. Since $N$ is Hausdorff and closed, such $K$ are also closed in $\hat A$.

Next suppose that $\pi$ satisfies (i) and (ii). Then $\pi$ belongs to the spectrum of a continuous-trace ideal $J$, and is a Fell point in $\hat J$; since $\hat J$ is open in $\hat A$, it trivially follows that $\pi$ is a Fell point in $\hat A$ and that $\hat J$ is an open Hausdorff neighbourhood of $\pi$. Now (ii) implies that $\hat J$ contains a closed neighbourhood, and $\pi\in V$. 
\end{proof}

\end{document}